\documentclass{tran-l}
\usepackage{amsmath,amssymb,amsfonts}

\newtheorem{thm}{Theorem}[section]
\newtheorem{cor}[thm]{Corollary}
\newtheorem{lem}[thm]{Lemma}
\newtheorem{prop}[thm]{Proposition}
\theoremstyle{definition}
\newtheorem{defn}[thm]{Definition}
\theoremstyle{remark}

\newtheorem{ex}[thm]{Example}
\numberwithin{equation}{section}

\begin{document}

\title[Envelopes and Weakly Radicals]
 {Envelopes and Weakly Radicals of Submodules}

\author{Erol Y{\i}lmaz}
\address{Abant \.{I}zzet Baysal University, Bolu, Turkey}
\email{yilmaz\_e2@ibu.edu.tr}

\author{Sibel Cansu}
\address{\.{I}stanbul Technical University, \.{Istanbul, Turkey}}
\email{sibel@gmail.com}



\begin{abstract}
Let $N$ be a submodule of a finitely generated module $M$ over a Noetherian ring. A method for the computation of the submodule generated by the envelope of $N$ is given. The relations between weakly prime submodules and their envelopes are investigated. Using these relations, a description of the weakly radical of a submodule is obtained. The results are illustrated by examples.

\end{abstract}

\maketitle

\section*{Introduction}
Throughout this paper all rings are commutative with identity and all modules
are unitary.

Let $R$ be a ring and $M$ be an $R$-module. A proper submodule $P$ of $M$ is said to be {\it primary submodule}  if whenever $rm \in P$ where $r \in R$ and $m \in M$ then $m \in P$ or $r^kM \subseteq N$ for some positive integer $k$.

Recall that $(P:M)=\{r \in R|rM\subseteq P\}$. If $P$ is a primary submodule of $M$ and $p=(P:M)$, then $P$ is called $p$-primary submodule (see \cite{mccasland}).

A {\it primary decomposition} of a submodule $N$ of $M$ is representation of $N$ as an intersection of finitely many primary submodule of $M$. Such a primary decomposition  $N=\cap_{i=1}^n Q_i$ with $p_i$-primary submodules $Q_i$ is called minimal if $p_i$'s are pairwise distinct and $Q_j \nsupseteq \cap_{i^\neq j}Q_i$ for all $j=1,\ldots,n$.

If $R$ is a Noetherian ring and $M$ is a finitely generated module, then any proper submodule $N$ has a minimal primary decomposition. The first uniqueness theorem states that for such a minimal primary decomposition the set of primes $\{p_1,\ldots,p_m\}$ is uniquely defined. These primes are called the associated primes of $M/N$. We denote this set by $Ass(M/N)$. It is clear that for any $p \in Ass(M/N)$, $(N:M) \subseteq p$.

The prime ideals in $Ass(M/N)$ that are minimal with respect to inclusion are called the isolated primes of $M/N$, the remaining associated prime ideals are the embedded primes of $M/N$.

The second uniqueness theorem states that not only the primes but also the primary components corresponding to isolated primes, the isolated components of $N$ in $M$, are uniquely defined. The other primary components, the embedded components of $N$ in $M$, need not be defined uniquely. The concepts and theorems about the primary decomposition of modules can be found in chapter 9 of \cite{sharp}.

The radical $\sqrt{I}$ of an ideal $I \subset R$ is characterized as the the set of elements $a \in R$ such that $a^n \in I$ for some positive integer $n$. The concept of envelope of a submodule is the generalization of this characterization to the modules. If $N$ is a submodule of an $R$-module $M$, then the envelope of $N$ in $M$ is defined to be the set
$$E_{M}(N)=\{rm: r\in R, m\in M \: and \: r^{k}m\in N \: for \:some \: k \in \mathbb{Z}^{+}\}.$$
Let $\langle E_M(N)\rangle$ be the submodule generated by the envelope. Although some methods for computing of radical of a submodule , which defined to be intersection of prime submodules containing $N$, are given in  \cite{marcelo} and \cite{smith}, it seems there is no description for the computation of the envelope in the literature. In section 1, we give a formula for the computation of $\langle E_M(N)\rangle$ if a minimal primary decomposition of $N$ is known. In this section, we use extensively the concepts and results from \cite{grabe}.

A proper submodule $N$ of an $R$-module $M$ is called \textit{a weakly prime submodule} if for each $m\in M$ and $a,b\in R$; $abm\in N$ implies that $am\in N$ or $bm\in N$. A proper submodule $N$ of an $R$-module $M$ is called \textit{a weakly primary submodule} if $abm\in N$ where $a,b\in R$ and $m\in M$, then either $bm\in N$ or $a^{k}m\in N$ for some $k\geq 1$. The concepts of weakly prime and weakly primary submodules are introduced a few years ago and they have been studied by some authors (for example see \cite{azizi1}, \cite{azizi2} and \cite{behboodi} ). In section 2, we investigated relations between weakly prime submodules and their envelopes. We also give an example to show  a conjecture given in \cite{behboodi} is false.

The weakly radical of a submodule $N$ of $M$, denoted by $wrad_M(N)$, is defined to be the intersection of all weakly prime submodules containing $N$. In \cite{azizi2}, a generalization of $\langle E_M(N) \rangle $ defined as follows: $E_0(N)=N, E_1(N)=E_M(N),E_2(N)=E_M(\langle E_M(N) \rangle)$, and for any positive integer $n$, it is defined $E_{n+1}(N)=E_M(\langle E_n(N)\rangle)$ inductively. $E_n(N)$ is called $n$-th envelope of $N$. Consider
$$
UE_M(N)=\bigcup_{n \in \mathbb{N}}\langle E_n(N)\rangle;
$$ $UE_M(N)$ is called the \textit{union of envelopes of} $N$. One can easily see that $N \subseteq \langle E_n(N) \rangle \subseteq UE_M(N) \subseteq wrad_M(N) \subseteq rad_M(N)$, for any $n \in \mathbb{N}$. If $UE_M(N)=rad_M(N)$( resp. $UE_M(N)=wrad_M(N)$), then it said to be radical formula (resp. weakly radical formula) holds for $N$.

A submodule $N$ is called a quasi-$p$-primary submodule in $M$, if $N$ has a unique isolated prime $p$ and possibly some embedded primes (see \cite{grabe}). In section 3, we show that weakly radical formula hold for quasi-primary submodules. Using this, we give a method to compute weakly radical of a submodule of a Noetherian module $M$.

\section{Envelope of Submodules}
Unless otherwise stated, after this point, we assume $R$ is a Noetherian ring, $M$ is finitely generated $R$- module and $N$ is proper submodule of $M$.

\begin{lem} \label{1}
Let $N=Q_{1}\cap Q_{2}\cap \cdots \cap Q_{k}$ be a minimal primary decomposition of $N$ where $\sqrt{Q_{i}:M}=p_{i}$ for all $i=1,2,\ldots ,k$. If $S=\{1,2,\ldots ,k\}$ and $ \emptyset \neq T \subsetneq S$, then
$$(\bigcap \limits_{i\in T} p_{i})(\bigcap \limits_{i\in S\setminus T} Q_{i})\subseteq \langle E_{M}(N) \rangle $$
\end{lem}
\begin{proof}
Let $n\in (\bigcap \limits_{i\in T} p_{i})(\bigcap \limits_{i\in S \setminus T} Q_{i})$. Then there exist $r_{j}\in \bigcap \limits_{i\in T} p_{i}$ and $m_{j}\in \bigcap \limits_{i\in S \setminus T} Q_{i}$ such that\\ $$n=r_{1}m_{1}+r_{2}m_{2}+\cdots +r_{s}m_{s}$$ for some $s\in \mathbb{Z}^{+}$.

\vspace{8pt}
Since ${{r}_{j}}\in \bigcap\limits_{i\in T}{{{p}_{i}}}$, ${{r}_{j}}^{{{k}_{j}}}M\subseteq \bigcap\limits_{i\in T}{{{Q}_{i}}}\text{ for some }{{k}_{j}}\in {{\mathbb{Z}}^{+}}$. In particular, ${{r}_{j}}^{{{k}_{j}}}m_{j}\in  \bigcap\limits_{i\in T}{{{Q}_{i}}}$ for all $j=1,2,\cdots, s$.

Since $m_{j}\in \bigcap \limits_{i\in S-T} Q_{i}$, ${{r}_{j}}^{{{k}_{j}}}m_{j}\in  \bigcap\limits_{i\in S-T}{{{Q}_{i}}}$ for all $j$. Thus we have ${{r}_{j}}^{{{k}_{j}}}m_{j}\in \bigcap\limits_{i=1}^{k}{{{Q}_{i}}}=N$ which means that $r_{j}m_{j}\in E_{M}(N)$ for all $j$. Thus $n\in \langle E_{M}(N) \rangle$.
\end{proof}

Before giving a formula for the envelope of a submodule in terms of its associated primes and primary submodules in its primary decomposition, we need some
technical prerequisites.

\begin{defn}
If $f \in R$ and $I$ is an ideal of $R$, then the set

$$ N:I^\infty=\{m \in M:  I^k m \subseteq N \: {\mathrm for \: some \: positive \: integer}\: k  \}$$

is called the \textit{stable quotient} of $N$ by $I$ in $M$.

\end{defn}
\begin{lem} \cite[Lemma 1]{grabe}\label{stable}
Let $P \subset M$ be a primary submodule of $M$ and $ f\in R$.

 $$(i) \: \: P:\langle f \rangle ^\infty=M \:\:{\mathrm if} \: \: f \in \sqrt{P:M}$$

$$(ii) \: \: P:\langle f \rangle ^\infty=P \: \: {\mathrm if} \: \:f \not \in \sqrt{P:M}$$

More generally, for arbitrary submodule $N$ of $M$ and its primary
decomposition \linebreak $N=\bigcap P_i$ into $p_i$-primary submodules $P_i$ we
get

$$ (iii) \: \: N: \langle f \rangle ^\infty =\bigcap_{f \not \in p_i}P_i $$

and for arbitrary ideal $I$ of $R$

$$ (iv) \: \: N:I^\infty =\bigcap_{I \not \subset p_i}P_i $$
\end{lem}

We can easily show that.

\begin{lem}\label{2}
Let $N$ be $p$-primary submodule of an $R$-module $M$. Then

$$(i) \: \: N: h =N, \:\:{\mathrm if} \: \: h \not \in p$$

$$(ii) \: \: N:h =M, \: \: {\mathrm if} \: \:h \in (N:M).$$
\end{lem}

The following theorem is the main result of this section.

\begin{thm}\label{envelope} With the notation in Lemma \ref{1}, $$\langle E_{M}(N) \rangle = N+(\bigcap\limits_{i=1}^{k} p_{i})M+\sum\limits_{\emptyset \neq T \subsetneq S} (\bigcap \limits_{i\in T} p_{i})(\bigcap \limits_{i\in S \setminus T} Q_{i}).$$

\end{thm}

\begin{proof}

Let $m\in \langle E_{M}(N) \rangle$. Then there exist $m_{j}\in M, r_{j}\in R$ such that $$m=r_{1}m_{1}+r_{2}m_{2}+\cdots +r_{t}m_{t}.$$ By the definition of $\langle E_{M}(N) \rangle$, $m_{j}\in N:\langle r_{j} \rangle^{\infty}$ for each $j=1,2,\ldots, t$.

For each $r_{j}$, either $r_{j}\in R \setminus \bigcup\limits_{i=1}^{k} p_{i}$ or there is a maximal proper subset $T$ of $S$ such that $r_{j}\in \bigcap\limits_{i\in T} p_{i}$.

If $r_{j}\in R \setminus \bigcup\limits_{i=1}^{k} p_{i}$, then $ N:\langle r_{j} \rangle^{\infty}=N$ by Lemma \ref{stable}. Hence $m_{j}\in N$ and so $r_{j}m_{j} \in N$.

If $r_{j}\in \bigcap\limits_{i\in T} p_{i}$, then $$N:\langle r_{j} \rangle^{\infty}= \bigcap_{i=1}^{k} (Q_i:\langle r_{j} \rangle^{\infty})= \bigcap\limits_{i\in S \setminus T} Q_{i}$$ by Lemma \ref{stable}. Hence $$r_{j}m_{j}\in (\bigcap \limits_{i\in T} p_{i})(\bigcap \limits_{i\in S \setminus T} Q_{i}).$$

If $r_{j}\in \bigcap\limits_{i=1}^{k} p_{i}=\sqrt{N:M}$, then $r_{j}m_{j} \in \sqrt{N:M}M$.

Thus we can conclude that $$ \langle E_{M}(N) \rangle \subseteq N+(\bigcap\limits_{i=1}^{k} p_{i})M+\sum\limits_{\emptyset \neq T \subsetneq S} (\bigcap \limits_{i\in T} p_{i})(\bigcap \limits_{i\in S \setminus T} Q_{i}).$$

For the other side of the inclusion, Lemma \ref{1} implies that $$\sum\limits_{\emptyset \neq T \subsetneq S} (\bigcap \limits_{i\in T} p_{i})(\bigcap \limits_{i\in S \setminus T} Q_{i})\subseteq \langle E_{M}(N) \rangle.$$ Moreover $N$ and $(\bigcap\limits_{i=1}^{k} p_{i})M=\sqrt{N:M}M$ are clearly in $\langle E_{M}(N) \rangle$.
\end{proof}

\begin{cor}
If $N$ is a $p$-primary submodule, then $$\langle E_{M}(N) \rangle = N+pM.$$
\end{cor}

Now we will give an application of Theorem \ref{envelope}. The computer algebra system {{\sc Singular} }was used for the computations (see \cite{DGPS}).
\begin{ex}
Let $R=\mathbb{Q}[x,y,z]$ and let $M=R\oplus R\oplus R$. Consider the
\vspace{8pt}
submodule $N= \langle xz\mathbf{e}_{3}-z\mathbf{e}_{1},x^{2}\mathbf{e}_{3},x^{2}y^{3}\mathbf{e}_{1}+x^2y^2z\mathbf{e}_{2}  \rangle$.

Primary decomposition of $N$ is $N=Q_{1}\cap Q_{2}\cap Q_{3}$ where
\vspace{10pt}

\begin{center}
$\begin{gathered}
  {Q_1} = \langle {{\mathbf{e}}_3},z{{\mathbf{e}}_1},y{{\mathbf{e}}_1} + z{{\mathbf{e}}_2},{z^2}{{\mathbf{e}}_2}\rangle {\text{ is }}\langle z\rangle -{\text{primary}}{\text{,}} \hfill \\
  {Q_2} = \langle {{\mathbf{e}}_1},{{\mathbf{e}}_3},{y^2}{{\mathbf{e}}_2}\rangle {\text{ is }}\langle y\rangle -{\text{primary and}} \hfill \\
  {Q_3} = \langle x{{\mathbf{e}}_1},x{{\mathbf{e}}_3} - {{\mathbf{e}}_1},{x^2}{{\mathbf{e}}_2}\rangle {\text{ is }}\langle x\rangle -{\text{primary.}} \hfill \\
\end{gathered} $
\end{center}
By Theorem \ref{envelope},
\vspace{12pt}

$\begin{gathered}
  \langle {E_M}(N)\rangle  = N + ({p_1} \cap {p_2} \cap {p_3})M + {p_1}({Q_2} \cap {Q_3}) + {p_2}({Q_1} \cap {Q_3}) + {p_3}({Q_1} \cap {Q_2}) \\
   + ({p_1} \cap {p_2}){Q_3} + ({p_1} \cap {p_3}){Q_2} + ({p_2} \cap {p_3}){Q_1}. \\
\end{gathered} $
\vspace{8pt}

It is clear that $({p_1} \cap {p_2} \cap {p_3})M=\langle xyz{{\mathbf{e}}_1},xyz{{\mathbf{e}}_2},xyz{{\mathbf{e}}_3}\rangle$. We also get
\vspace{8pt}

\begin{center}
$\begin{gathered}
{p_1}({Q_2} \cap {Q_3})= \langle xz{{\mathbf{e}}_1},xz{{\mathbf{e}}_3}-z{{\mathbf{e}}_1},x^2y^2z{{\mathbf{e}}_2}\rangle \hfill \\
{p_2}({Q_1} \cap {Q_3})= \langle xyz{{\mathbf{e}}_3}-yz{{\mathbf{e}}_1},x^2y{{\mathbf{e}}_3},x^2y^2{{\mathbf{e}}_1}+x^2yz{{\mathbf{e}}_2}\rangle \hfill \\
{p_3}({Q_1} \cap {Q_2}) =  \langle x{{\mathbf{e}}_3},xz{{\mathbf{e}}_1},xy^3{{\mathbf{e}}_1}+xy^2z{{\mathbf{e}}_2}\rangle  \hfill \\
({p_1} \cap {p_2}){Q_3}= \langle xyz{{\mathbf{e}}_1},xyz{{\mathbf{e}}_3}-yz{{\mathbf{e}}_1},x^2yz{{\mathbf{e}}_2}\rangle \hfill \\
({p_1} \cap {p_3}){Q_2}= \langle xz{{\mathbf{e}}_1},xz{{\mathbf{e}}_3},xy^2z{{\mathbf{e}}_2}\rangle \hfill \\
({p_2} \cap {p_3}){Q_1}=  \langle xy{{\mathbf{e}}_3},xyz{{\mathbf{e}}_1},xy^2{{\mathbf{e}}_1}+xyz{{\mathbf{e}}_2},xyz^2{{\mathbf{e}}_2} \rangle  \hfill \\
\end{gathered} $
\end{center}

Thus $$\langle E_{M}(N)\rangle= \langle z{{\mathbf{e}}_1},x{{\mathbf{e}}_3},xyz{{\mathbf{e}}_2},xy^2{{\mathbf{e}}_1} \rangle.  $$

\end{ex}

\begin{cor}\label{3}
If $\langle E_{M}(N) \rangle =N$, then each isolated component of primary decomposition of $N$ must be prime.
\end{cor}

\begin{proof}
Let $N=Q_{1}\cap Q_{2}\cap \cdots \cap Q_{n}$ with $Q_{i}$'s are $p_{i}$-primary submodules. Let $Q_{k}$ be one of the isolated components of $N$. If $Q_{k}$ were not a prime submodule, then there would be exist $x\in p_{k} \setminus (Q_{k}:M)$. Hence there exists $m\in M$ such that $xm\not \in Q_{k}$. Since $p_{k}$ is an isolated prime, we can find an element $y\in (\bigcap\limits_{j\neq k}p_{j}) \setminus p_{k}$. Then $$xym\in (\bigcap\limits_{j=1}^{n}p_{j})M\subseteq \langle E_{M}(N) \rangle =N\subseteq Q_{k}.$$ Since $Q_{k}$ is $p_{k}$-primary and $xm\not \in Q_{k}$, $y\in p_{k}$ which is a contradiction.
\end{proof}

In general for submodules $N_1$ and $N_2$ of a module $M$, $\langle E_M(N_1 \cap N_2)\rangle \neq \langle E_M(N_1)\rangle \cap \langle E_M(N_2)\rangle$. We would like to give a condition for submodules under which we have the equality.

\begin{defn} A submodule $N$ is called a quasi-$p$-primary submodule
in $M$, if $N$ has a unique isolated prime $p$ (and possibly embedded primes).
\end{defn}

The following proposition is crucial for the computing primary decomposition and is quite useful for our purpose.

\begin{prop}  \cite[Proposition 1]{grabe}
Assume that $L =\{p_1,\ldots,p_k\}$ are the isolated primes of $N$. For $i, j = 1,\ldots,m$ take $f_i \in R$ such that $f_i \in  p_j$ if $i\neq j$, but $f_i \not \in p_i$ and take integers $e_i$ such that $f_{i}^{e_i} N_i \subset N$.

Then:

(i)  $N_i$ is a quasi-$p_i$-primary module in $M$.

(ii) The sets $A_i = Ass(M/Ni)= \{p \in  Ass(M/N) : f_i \not \in p\}$ are pairwise disjoint.

(iii) For $J := \langle f_1,f_2,\ldots,f_k \rangle$  we have $$N = (\bigcap Ni)\cap (N + JM )$$
This is a decomposition of $N$ into quasi-primary components $N_i$ and
a component $N' := N + J M \subset M$ of lower (relative) dimension.
\end{prop}

\begin{thm}\label{onlyquasi}
Assume that $L =\{p_1,\ldots,p_k\}$ are the isolated primes of $N$ and the minimal primary decomposition of $N$ contains only quasi-primary components $N_i$ for $i=1,\ldots,k$. If $ \langle E_M(N) \rangle=N$, then $\langle E_M(N_i)\rangle=N_i$ for each quasi-primary component $N_i$. Hence $$\langle E_M(N)\rangle= \langle E_M \big( \bigcap_{i=1}^k (N_i)\big)\rangle=\bigcap_{i=1}^k \langle E_M(N_i)\rangle.$$
\end{thm}

\begin{proof}
For a fixed $i$, let $Ass(M/N_i)=\{p_{i_1}=p_i,p_{i_2}, \ldots,p_{i_{s_i}}\}$ and $p_i \subseteq p_{i_k}$ for every $k$ and let $N_i=Q_{i_1} \cap \cdots \cap Q_{i_{s_i}}$ where each $Q_{i_k}$ is $p_{i_k}$-primary. By the Theorem \ref{envelope},
$$
\langle E_M(N) \rangle= N+ (\bigcap_{i=1}^k p_i)M+\sum_{\emptyset \neq T \varsubsetneq S} \big(\bigcap_{j \in T}p_{i_j} \big)\big (\bigcap_{j \in S\setminus T} Q_{i_j}\big )
$$
and
$$
\langle E_M(N_i)\rangle =N_i+p_iM+\sum_{T\varsubsetneq S_i}\big(\bigcap_{r \in T}p_{i_r} \big)\setminus \big(\bigcap_{r \in S_i\setminus T}Q_{i_r} \rangle
$$
where $S_i=\{i_1,i_2,\ldots,i_{s_i}\}$ and $S=\bigcup_{i=1}^k S_i$.

Let $x \in p_i$ and $m \in M$. Take $y=\big(\bigcap_{j \neq i}p_j \big)\setminus \big(\bigcup_{t=2}^{s_i}p_{i_t}\big)$. Then
$$
yxm \in \big(\bigcap_{j=1}^kp_j \big)M \subseteq \langle E_M(N) \rangle \subseteq Q_{i_t}$$ for $t=1,\ldots,s_i$. Since $Q_{i_t}$ is primary and $y \not \in p_{i_t}$, $xm \in Q_{i_t}$. Hence $xm \in N_i$.

Now let $x \in \bigcap_{r \in T}, m \in \bigcap_{S_i \setminus T}Q_{i_r}$ for some $T\varsubsetneq S_i$.  Take $$y=\big(\bigcap_{j \neq i}p_j \big)\setminus \big(\bigcup_{t=2}^{s_i}p_{i_t}\big).$$ Then
$$
yxm \in \bigg[ \big(\bigcap_{ j \neq i}p_j \big) \cap \big(\bigcap_{r \in T}p_{i_t} \big) \bigg]\big(\bigcap_{r\in S_i\setminus T}Q_{i_r}\big)
.$$
Since $$ \bigcap_{j \neq i}p_j = \bigcap_{ j\neq i} \bigcap_{t=1}^{s_j}p_{j_t},$$
$$
\bigg[ \big(\bigcap_{ j \neq i}p_j \big) \cap \big(\bigcap_{r \in T}p_{i_t} \big) \bigg]\big(\bigcap_{r\in S_i\setminus T}Q_{i_r}\big)\subseteq \langle E_M(N)\rangle \subseteq N_i.
$$ 
Thus $yxm \in Q_{i_t}$ for $t=1,\ldots,s_i$. Since $Q_{i_t}$ is primary and $y \not \in p_{i_t}$, $xm \in Q_{i_t}$ and hence $xm \in N_i$. Therefore $\langle E_M(N_i)\rangle=N_i$ and the conclusion easily follows. 
\end{proof}



\section{Weakly Prime Submodules}
In this section we investigate the relations between weakly prime submodules and their envelopes.
\begin{lem}\label{weakenv}
If $N$ is a weakly prime submodule, then $\langle E_{M}(N) \rangle =N$.
\end{lem}

\emph{Proof.} Let $x\in \langle E_{M}(N) \rangle$. Then there exist elements $r_{i}\in R$ and $m_{i}\in M$
$(1\leq i \leq k)$ such that
\begin{eqnarray*}
       x= r_{1}m_{1}+\cdots +r_{k}m_{k}& & with \quad  r_{i}^{t_{i}}m_{i}\in N
     \end{eqnarray*}
for some $t_{i}\in \mathbb{Z}^{+}$. Since $N$ is weakly prime, $r_{i}^{t_{i}}m_{i}\in N$ implies that $r_{i}m_{i}\in N$ or $r_{i}^{t_{i}-1}m_{i}\in N$. If $r_{i}m_{i}\in N$, then $ x= r_{1}m_{1}+\cdots +r_{k}m_{k} \in N$. If $r_{i}^{t_{i}-1}m_{i}\in N$, then $r_{i}m_{i}\in N$ or $r_{i}^{t_{i}-2}m_{i}\in N$. By the same process, $r_{i}m_{i}\in N$ for all cases. Hence $x\in N$, which means that $\langle E_{M}(N) \rangle \subseteq N$. Other side of the inclusion is obvious. $\Box$

\begin{thm} \label{wprime}
  Suppose that $N=Q_{1}\cap Q_{2}\cap \cdots \cap Q_{s}$ where each $Q_{i}$ is $p_{i}$-primary submodule with $p_{1} \subset p_{2}\subset \cdots \subset p_{s}$. If $E_M(N)=N$, then $N$ is a weakly prime submodule.
\end{thm}
\emph{Proof.} Since $p_{1} \subset p_{2}\subset \cdots \subset p_{s}$, by the Theorem~\ref{envelope}
$$N=\langle E_{M}(N) \rangle =N+p_{1}M+\sum \limits _{i=2}^{s} p_{i}(\bigcap \limits _{j=1}^{i-1} Q_{j}).$$
Let $abm\in N$ with $a,b\in R$ and $m\in M$. Let $i$ be the first index for which $m \not \in Q_i$. Since $Q_i$ is $p_i$-primary, $ab \in p_i$ and so either $a \in p_i$ or $b \in p_i$. If $i=1$, then since $p_1 M \subset \langle E_M(N)\rangle =N$, either $am \in N$ or $bm \in N$. Let $i>1$. Since $p_{i}(\bigcap _{j=1}^{i-1} Q_{j}) \subset \langle E_M(N)\rangle =N$, either $am \in N$ or $bm \in N$. Hence $N$ is a weakly prime submodule. $\Box$

 The following conjecture is stated in \cite{behboodi}: Let $R$ be a ring and $M$ be an $R$-module. Then for every weakly primary submodule $Q$ of $M$, $\langle E_{M}(Q) \rangle$ is a weakly prime submodule. Notice that they use the notation $\sqrt[nil]{Q}$ for $\langle E_{M}(Q) \rangle$ in \cite{behboodi}.

The following example shows that the conjecture is false.

\begin{ex}
Let $R=\mathbb{Q}[x,y]$ and let $M=R\oplus R$. Consider the
\vspace{8pt}
submodule $N= \langle x\mathbf{e}_{1}+y^3\mathbf{e}_{2},x^{2}\mathbf{e}_{1},x\mathbf{e}_{2} \rangle$.
One can easily see that $(N:M)= \langle x^2 \rangle$ and $N$ is $\langle x \rangle$-primary submodule. Hence
$$\langle E_M(N) \rangle = N+ \langle x \rangle M= \langle  x\mathbf{e}_{1}, x\mathbf{e}_{2}, y^3 \mathbf{e}_{2} \rangle$$
Then $\langle E_M(N) \rangle$ is not weakly prime submodule since $y^2 (0,y)=(0,y^3) \in \langle E_M(N)\rangle $ but $y (0,y)=(0,y^2) \not \in \langle E_M(N)\rangle $.
\end{ex}

If we weaken the conditions of the conjecture as follows, then we can obtain the desired result.
\begin{cor}\label{wp}
Let $R$ be a Noetherian ring and $M$ be a finitely generated $R$-module. Then for every weakly primary submodule $Q$ of $M$; if $\langle E_{M}(Q) \rangle=Q$, then $Q$ is weakly prime.
\end{cor}
\emph{Proof.}  Let $Q=Q_{1}\cap Q_{2}\cap \cdots \cap Q_{k}$ be primary decomposition of $Q$ with $\sqrt{Q_{i}:M}=p_{i}$ $(1\leq i\leq k)$. By \cite[Proposition 3.1]{behboodi}, $p_{1}\subset p_{2}\subset \cdots \subset p_{k}$. Then Theorem \ref{wprime} implies that $Q$ is weakly prime submodule. $\Box$
\begin{cor}
 Let $N=Q_{1}\cap Q_{2}$ be a submodule of $M$ where $Q_{i}$ is $p_{i}$-primary. If $\langle E_{M}(N) \rangle=N$, then either $Q_{1}$ and $Q_{2}$ are both prime or $N$ is weakly prime.
\end{cor}
\emph{Proof.} We have two cases: $p_{1}\nsubseteq p_{2}$ or $p_{1}\subseteq p_{2}$. If $p_{1}\nsubseteq p_{2}$, then both $p_{1}$ and $p_{2}$ are isolated primes. From  Corollary \ref{3}, $Q_{1}$ and $Q_{2}$ are prime submodules. If $p_{1}\subseteq p_{2}$, then  Theorem~\ref{wprime} implies that $N$ is weakly prime.  $\Box$

\begin{lem}\label{quasi}
If $N$ is a quasi-$p_1$-primary submodule and $\langle E_M(N)\rangle=N$, then $N$ can be expressed as an intersection of finitely many weakly prime submodules containing $N$.
\end{lem}

\emph{Proof.}
Let $Ass(M/N)=\{p_1,\ldots,p_s\}$ and $S=\{1,\ldots,s\}$. If $N$ contains only one maximal associated prime with respect to inclusion, then its associated primes form a chain $p_1 \subset \cdots \subset p_s$. Hence $N$ is weakly prime by Theorem~\ref{wprime}.

Suppose that $N$ has more than one maximal element. For each maximal $p_{j}$, we have a unique chain of associated primes $p_1=p_{j_1} \subset p_{j_2}\subset \cdots \subset p_{j_t}=p_j$. Let $N_j=Q_{j_1} \cap Q_{j_2} \cdots \cap Q_{j_t}$ where $Q_{j_1}=Q_1$ and $Q_{j_t}=Q_j$. From Theorem~\ref{envelope},
$$\langle E_{M}(N) \rangle = N+ p_{1}M+\sum\limits_{T\subset S} (\bigcap \limits_{i\in T} p_{i})(\bigcap \limits_{i\in S \setminus T} Q_{i})$$ and
$$\langle E_{M}(N_j) \rangle= N_j+p_{1}M+\sum \limits _{i=2}^{t} p_{j_i}(\bigcap \limits _{k=1}^{i-1} Q_{j_k}). $$ Our aim to show that $\langle E_M(N_j)\rangle=N_j$.
Clearly $p_{1}M \subset \langle E_M(N)\rangle=N \subset N_j$. Let $B=Ass(M/N)\setminus Ass(M/N_j)$. Take $x \in p_{j_i}$ and $ m \in \bigcap_{k=1}^{i-1} Q_{j_k}$. Since $p_j$ is a maximal prime and associated primes pairwise distinct, there exists $ y\in (\bigcap\limits_{p \in B} p)\setminus p_{j}$. Hence $$yxm \in (p_{j_i} \cap (\bigcap\limits_{p \in B} p ) (\bigcap\limits_{k=1}^{i-1} Q_{j_k}) \subset \langle E_M(N)\rangle=N \subset N_j \subset Q_{j_k}.$$ Since each $Q_{j_k}$ is $p_{j_k}$-primary and $ y \not \in p_{j_k}$, $xm \in Q_{j_k}$. Hence $xm \in N_j$. This implies $\langle E_M(N_j)\rangle=N_j$ and $N_j$ is weakly prime by Theorem~\ref{wprime}. Since $N=\cap N_j$, $N$ is intersection of finitely many weakly prime submodules.
$\Box$

Using the previous Lemma and Theorem~\ref{onlyquasi}, we can conclude the following.
\begin{thm}
Assume that $L =\{p_1,\ldots,p_k\}$ are the isolated primes of $N$ and  the minimal primary decomposition of $N$ contains only quasi-primary components $N_i$ for $i=1,\ldots,k$. If $\langle E_M(N)\rangle =N$, then $N$ can be expressed as the intersection of finitely many weakly prime submodules.
\end{thm}

\begin{defn}
A proper submodule $N$ of an $R$-module $M$ is called semiprime if whenever $r^km \in N$ for some $r \in R,m \in M$ and natural number $k$, then $rm \in N$.
\end{defn}

The question when a semiprime module can be expressed as a finite intersection of weakly prime submodules discussed in \cite {behboodi2}. We have the following contribution to this discussion.

\begin{lem}\label{semiprime}
Let $N$ be a semiprime submodule of an $R$-module $M$. Then $\langle E_{M}(N) \rangle=N$.
\end{lem}
\emph{Proof.} Let $x\in \langle E_{M}(N) \rangle$. Then there exist elements $r_{i}\in R$, $m_{i}\in M$ $(1 \leq i \leq k)$ such that
\begin{eqnarray*}
       x= r_{1}m_{1}+\cdots +r_{k}m_{k}& & with \quad  r_{i}^{t_{i}}m_{i}\in N
     \end{eqnarray*}
for some $t_{i}\in \mathbb{Z}^{+}$. Since $N$ is semiprime, $r_{i}m_{i}\in N$ for all $i$. Hence $x\in N$ and $\langle E_{M}(N) \rangle=N$.
 $\Box$
\begin{cor}
Let $R$ be a Noetherian ring and $M$ be a finitely generated $R$-module. Each semiprime submodule $N$ of $M$ is intersection of weakly prime submodules, if the primary decomposition of $N$ contains only quasi-primary components for each isolated prime of $N$.
\end{cor}
\

We have also the following result.
\begin{prop}
Let $N$ be a weakly primary submodule of $M$. Then $N$ is semiprime if and only if $N$ is weakly prime.
\end{prop}
\emph{Proof.} Suppose $N$ is semiprime. Then by Lemma \ref{semiprime}, $\langle E_{M}(N) \rangle=N$. Since $N$ is weakly primary, $N$ is weakly prime by Corollary \ref{wp}.

Conversely assume $N$ is weakly prime. Let $r\in R, m\in M$ and $r^{k}m\in N$ for some $k\in \mathbb{Z}^{+}$. $r^{k}m\in N$ implies that $rm\in N$ or $r^{k-1}m\in N$. If $rm\in N$, then $N$ is semiprime. If $r^{k-1}m\in N$, then by the same process $rm\in N$. Hence for all cases $N$ is semiprime submodule. $\Box$

\section{Weakly Radical}

The next two results are quiet useful for our purpose.

\begin{lem} \cite[Lemma 2.3]{smith} \label{quot}
For every prime ideal $p$ of $R$ such that $(N:M) \subseteq p$, $(N+pM:M)=p$.
\end{lem}

\begin{cor}\cite[Corollary 2.4]{smith}
$p \in Ass(M/(N+pM))$ if and only if $(N:M) \subseteq p$.
\end{cor}

We can generalize this results.

\begin{lem} \label{quasienv}
If $(N:M)=p_1$ for a prime ideal $p_1$, then $(\langle E_M(N) \rangle :M )=p_1$.
\end{lem}

\emph{Proof.}
Since $(N:M)=p_1$ and every associated prime of $N$ contains $(N:M)$, $N$ is a quasi-$p_1$-primary submodule. Suppose $N=Q_1\cap \cdots \cap Q_s$ is a minimal primary decomposition of $N$ where each $Q_i$ is $p_i$-primary. Let $S=\{1,\ldots,s\}$ and let $T$ be a non-empty proper subset of $S$. Clearly, $N$ is a quasi-$p_1$-primary submodule. Since all associated primes of $N$ contain $p_1$, $\sqrt{N:M}=p_1$ and $(\bigcap \limits_{i\in T} p_{i})(\bigcap \limits_{i\in S-T} Q_{i}) \subset p_1 M$ when $ 1 \in T$. Hence
$$
\langle E_M(N) \rangle=N+p_1 M+\sum\limits_{1 \not \in T \subsetneq S} (\bigcap \limits_{i\in T} p_{i})(\bigcap \limits_{i\in S \setminus T} Q_{i}).
$$

  Since $(\bigcap \limits_{i\in S\setminus T} Q_{i}) \subset Q_1$ when $1 \not \in T$, $$\big((\bigcap \limits_{i\in T} p_{i})(\bigcap \limits_{i\in S-T} Q_{i})\big):M \subseteq Q_1:M \subseteq p_1.$$

Let
$$
K=\big (N+\sum\limits_{1 \not \in T\subsetneq S} (\bigcap \limits_{i\in T} p_{i})(\bigcap \limits_{i\in S \setminus T} Q_{i})\big )
.$$
Then $\langle E_M(N)=K+p_1M \rangle$ where $K:M \subseteq p_1$. By Lemma \ref{quot}, $(\langle E_M(N) \rangle :M )=p_1$.
$\Box$

The following concept is crucial in computation of radical of a submodule.

\begin{defn}
Let $p$ be any prime ideal of $R$. Following \cite{smith}, we let $cl_p(N)$ denote the $p$-closure of $N$, as defined by
$$
cl_p(N)=\{m \in M: rm \in N \: \mathrm{for} \: \mathrm{some} \: r \in R/p\}
$$
\end{defn}
It is clear then $cl_p(N)=\cup_{r \in R\setminus p} (N:r)$ and $ N \subseteq cl_p(N)$. The most interesting case is where $(N:M) \subseteq p$. In fact, if a minimal primary decomposition of $N$ is known, then $cl_p(N)$ can be computed as an intersection of certain primary submodule of the primary decomposition.
\begin{lem}\label{closure}
Let $p$ be prime ideal such that $(N:M) \subseteq p$. If $N=Q_1 \cap \cdots \cap Q_s$ is a minimal primary decomposition with $p_i$-primary submodule $Q_i$'s, then
$$
cl_p(N)=\bigcap_{p_i \subseteq p} Q_i
$$

\end{lem}

\emph{Proof.}
Let $r \in R \setminus p$. Then by Lemma \ref{2}
$$
(N:r)=\bigcap_{i=1}^s (Q_i:r)= \big ( \bigcap_{p_i \nsubseteq p}(Q_i:r) \big ) \cap \big( \bigcap_{p_i \subseteq p} Q_i \big).
$$
If $p_i \subseteqq p$ for all $i$'s, then we obtain the result. If $p_i \nsubseteq p$ for some $i$, then there exists $r_i \in (Q_i:M)\setminus p$. Let $r_0=\prod_{p_i \nsubseteqq p} r_i$. Since $r_0 \in (Q_i:M)$ for each $i$ satisfying $ p_i \nsubseteqq p$, Lemma \ref{2} implies
$$
(N:r_0)=\big ( \bigcap_{p_i \nsubseteq p}(Q_i:r_0) \big ) \cap \big( \bigcap_{p_i \subseteq p} Q_i \big)= \bigcap_{p_i \subseteq p}Q_i.
$$
The conclusion is obvious.
$\Box$

In \cite{smith}, the authors just compute $cl_p(N+pM)$ which is the only closure needed in the computation of the radical. The above lemma gives a method of computation  $cl_p(N)$ for any submodule $N$ whose primary decomposition is known.

Before defining similar concept for the computation of weakly radical, we need the following result.

\begin{thm}\label{formula}
If $(N:M)=p$ for a prime ideal $p$ of $R$, then the weakly radical formula hold for $N$.
\end{thm}
\emph{Proof.}
Since $(N:M)=p$, $N$ and $\langle E_M(N) \rangle$ are  quasi-$p$-primary submodules by Lemma \ref{quasienv}. Furthermore $\langle E_n(N) \rangle$ is also quasi-$p$-primary for any $ n\in \mathbb{Z}^+$ by the same reason. Since we assumed $M$ is a finitely generated module over a Noetherian ring $R$, $UE_M(N)=\langle E_k (N) \rangle$ for some $k\in \mathbb{Z}^+$. That means $\langle E_k (N) \rangle= \langle E_t (N) \rangle$ for every $ t \geq k$. By Lemma \ref{quasi}, $UE_M(N)$ can be expressed as an intersection of finitely many weakly prime submodules containing $N$. Therefore $wrad_M(N)=UE_M(N)$.
$\Box$

\begin{defn}
Let $(N:M)=p$ for a prime ideal $p$ of $R$. We let $wcl_p(N)$ denote the weakly $p$-closure of $N$, as defined by $wcl_p(N)=UE_M(N)$.
\end{defn}

At this point we would like to emphasize that the associated primes of a weakly primary submodules should form a chain according to \cite[Proposition 3.1]{behboodi}. Since every weakly prime submodule is also weakly primary, the associated primes of the prime submodules also satisfied this property. Hence all weakly prime submodules have a unique isolated prime.

\begin{lem}
If $P$ is a weakly prime submodule containing a submodule $N$ and if $p_1$ is the isolated prime of $P$, then $p_1 \supseteq (N:M)$.
\end{lem}
\emph{Proof.}
Let $P$ be a weakly prime submodule with the isolated prime $p_1$. Therefore $E_M(P)=P$ by Lemma \ref{weakenv}. If $P=Q_1 \cap \cdots \cap Q_s$ is the minimal primary decomposition where each $Q_i$ is $p_i$-primary, then $Q_1$ is $p_1$-prime submodule by Corollary \ref{3}. Hence $(N:M) \subseteq (P:M) \subseteq (Q_1:M)=p_1$.
$\Box$

Hence when computing weakly radical of a submodule, we can restrict ourself to the weakly prime submodules whose isolated primes contain $N:M$.

Given Lemma \ref{quot} and Theorem \ref{formula}, $wcl_p(N+pM)$ is of particular interest.
Of course, if $(N:M) \subseteq p$, $N \subseteq N+pM \subseteq UE(N+pM)=wcl_{p}(N+pM)=wrad_M(N+pM)$. Hence $wcl_{p}(N+pM)$ can be expressed as an intersection of weakly prime submodules containing $N$. The next theorem shows that these weakly prime submodules are minimal among the weakly prime submodules with isolated prime $p$ containing $N$.
\begin{thm}\label{wcl1}
Let $P $ be weakly prime submodule containing $N$ and let $p_1$ be the isolated prime of $P$, then $wcl_{p_1}(N+p_1M) \subseteq P$.
\end{thm}
\emph{Proof.}
Since $P$ is weakly prime submodule with the isolated prime $p_1$, $E_M(P)=P$ and its associated primes forms a chain $p_1 \subset p_2 \cdots \subset p_s$. Since $\langle E_M(P) \rangle = P+p_1M+ \cdots $, $N+p_1M \subset P+p_1M \subset E_M(P)=P$. Since $(N+p_1M):M=p_1$, the the weakly radical formula hold for $N+p_1M$. Then $$wcl_{p_1}(N+p_1M)=UE_M(N+p_1M)\subset UE_M(P)=P.$$
$\Box$

\begin{prop}\label{wcl2}
If $P_1$ and $P_2$ are weakly primary submodules containing $N$ with isolated primes $p_1$ and $p_2$ respectively and $p_1 \subseteq p_2$, then $wcl_{p_1}(N+p_{1}M) \subseteq wcl_{p_2}(N+p_{2}M)$.
\end{prop}
\emph{Proof.}
Since $p_1 \subseteq p_2$, $N+p_1M \subseteq N+p_2M$. Then clearly $\langle E_M(N+p_{1}M) \rangle \subseteq \langle E_M(N+p_{2}M) \rangle$. Hence the conclusion is obvious.
$\Box$

If $R$ is a Noetherian ring and $I \subset R$ is an ideal, the set of associated primes of $I$ is the set  $Ass(I)=\{P \subset R | P \: \mathrm{prime} \: P = I:\langle b\rangle \: \mathrm{for} \: \mathrm{some} \: b \in R \}$.  The set of associated primes which are minimal with respect to set inclusion is denoted by $minAss(I)$. Hence using Theorem \ref{wcl1} and Proposition \ref{wcl2} we can give the following formula for the computation of $wrad_M(N)$.
\begin{cor}
$$wrad_M(N)=\bigcap\limits_{p \in minAss((N:M))} wcl_p(N+pM).$$
\end{cor}

Now we illustrate the computation of the weakly radical of a submodule by an example. We again use the computer algebra system {{\sc Singular} } for the computations (see \cite{DGPS}).

\begin{ex}
Let $R=\mathbb{Q}[x,y,z]$ and let $M=R\oplus R\oplus R$. Consider the submodule
$$
N=\langle x^2\mathbf{e_1}+y^2\mathbf{e_2},x^2 z \mathbf{e_2},y^3 z \mathbf{e_1}+z^3\mathbf{e_3} \rangle.
$$

$Ass(M/N)=\{\langle z \rangle, \langle x \rangle \}$.
\smallskip

$W_1=N+\langle z \rangle M= \langle z \mathbf{e_1}, z\mathbf{e_2}, z\mathbf{e_3}, x^2\mathbf{e_1} +y^2 \mathbf{e_2} \rangle $.
\smallskip

Since $W_1$ is $\langle z \rangle$-prime, $wcl_{\langle z \rangle}(W_1)=UE_M(W_1)=W_1$.
\smallskip

$W_2=N+ \langle x \rangle M= \langle x\mathbf{e_1}, x\mathbf{e_2}, x\mathbf{e_3},y^2 \mathbf{e_2},y^3z \mathbf{e_1}+z^3 \mathbf{e_3} \rangle$.
\smallskip

$Ass(M/W_2)=\{p_1=\langle x \rangle,p_2= \langle x,z \rangle,p_3= \langle x,y \rangle\}$.
\smallskip

The primary decomposition of $W_2=Q_1 \cap Q_2 \cap Q_3$ where
\smallskip

$Q_1=\langle \mathbf{e_2}, x\mathbf{e_1}, x \mathbf{e_3}, y^3\mathbf{e_1}+z^2\mathbf{e_3} \rangle,$
\smallskip

$Q_2= \langle \mathbf{e_2}, z \mathbf{e_1}, z\mathbf{e_3}, x\mathbf{e_1}, x\mathbf{e_3} \rangle$ and
\smallskip

$Q_3= \langle \mathbf{e_3}, x\mathbf{e_1}, x\mathbf{e_2},y^2 \mathbb{e_1},y^2 \mathbf{e_2} \rangle$. Here each $Q_i$ is $p_i$-primary.

Using the Theorem \ref{envelope}, one can compute that
$$
\langle E_M(W_2) \rangle=\langle y\mathbf{e_2}, x\mathbf{e_1}, x\mathbf{e_2}, x\mathbf{e_3},y^3z \mathbf{e_1}+z^3\mathbf{e_3} \rangle.
$$

In fact, $\langle E_M( \langle E_M(W_2) \rangle )\rangle= \langle E_M(W_2) \rangle$. Therefore $wcl_{\langle x \rangle}(W_2)=\langle E_M(W_2) \rangle$.

Thus

\begin{eqnarray*}
 wrad_M(N) &=& wcl_{\langle z \rangle}(W_1) \cap wcl_{\langle x \rangle}(W_2) \\
   &=&  \langle yz \mathbf{e_2}, xz\mathbf{e_1}, xz \mathbf{e_2}, xz \mathbf{e_3}, x^2 \mathbf{e_1}+y^2\mathbf{e_2},y^3z \mathbf{e_1}+z^3\mathbf{e_3} \rangle.
\end{eqnarray*}

\end{ex}


\end{document}